\documentclass[12pt]{article}
\usepackage[utf8]{inputenc}
\usepackage{amsmath,amsfonts,amssymb,amsthm}
\usepackage{bbm}
\usepackage{mathtools}
\usepackage{graphicx}
\usepackage{tabularx}
\usepackage{color}
\definecolor{darkgreen}{rgb}{0,0.55,0}
\newcommand{\grad}{\nabla}

\renewcommand{\div}{\grad\cdot}
\newcommand{\curl}{\grad\times}

\newcommand{\N}{\mathbbm{N}}

\newcommand{\R}{\mathbbm{R}}

\DeclareMathOperator{\dist}{dist}

\DeclareMathOperator{\spt}{spt}

\usepackage{sectsty}

\sectionfont{\large}
\subsectionfont{\normalsize}
\subsubsectionfont{\normalsize}
\paragraphfont{\normalsize}

\newcommand{\eps}{\varepsilon}

\def\XXint#1#2#3{{\setbox0=\hbox{$#1{#2#3}{\int}$ }
\vcenter{\hbox{$#2#3$ }}\kern-.59\wd0}}

\newtheorem{theorem}{Theorem}
\newtheorem{lemma}{Lemma}

\newtheorem{cor}{Corollary}

\title{Euler vortex dynamics for unbounded vorticities}

\author{Stefano Ceci \and Christian Seis}
\date{Institut f\"ur Analysis und Numerik, Westf\"alische Wilhelms-Universit\"at M\"unster, Germany.}

\begin{document}
\phantom{ }
\vspace{4em}

\begin{flushleft}
{\large \bf On the dynamics of point vortices for the 2D Euler equation with $L^p$ vorticity}\\[2em]
{\normalsize \bf Stefano Ceci and Christian Seis}\\[0.5em]
\small Institut f\"ur Analysis und Numerik,  Westf\"alische Wilhelms-Universit\"at M\"unster, Germany.\\
E-mails: ceci@wwu.de, seis@wwu.de \\[3em]
{\bf Abstract:} 
We study the evolution of  solutions to the 2D Euler equations whose vorticity is sharply concentrated in the Wasserstein sense around a finite number of points. Under the assumption that the vorticity is merely $L^p$ integrable for some $p>2$, we show that the evolving vortex regions remain concentrated around points, and these points are close to solutions to the Helmholtz--Kirchhoff point vortex system.
\end{flushleft}

\vspace{2em}

\section{Introduction}

This article is concerned with the study of the interaction of  vortices in an ideal incompressible fluid contained in the two-dimensional plane. We examine
the evolution of sharply concentrated vortex regions that are driven by the Euler equations and we analyze their convergence to a point vortex system if the vortex regions shrink to a collection of points.

The study of the relation between the PDE description of vortex dynamics given by the Euler equations and the ODE description given by a point vortex system has a long history. 
The earliest research in this direction goes back to Helmholtz \cite{Helmholtz1858}, who in the middle of the 19th century derived a model for the interaction of point vortices from the Euler equations via formal asymptotics. His work was complemented by that of Kirchhoff \cite{Kirchhoff76}, who showed that this point vortex system   is Hamiltonian.
We will refer to this model accordingly as the Helmholtz--Kirchhoff
point vortex system.

On a  rigorous level, a derivation of the point vortex system from the Euler equation was achieved for the first time by Marchioro and Pulvirenti in the 1980ies \cite{MarchioroPulvirenti83} and was improved in many subsequent works, e.g.~\cite{MarchioroPulvirenti93, CapriniMarchioro15, ButtaMarchioro18}. In these works, the authors consider bounded vorticity solutions to the Euler equations whose support is confined in separated regions of small diameter, and prove that the vortex centers are during their evolution close  to the point vortices described by the Helmholtz--Kirchhoff system.

In a recent work \cite{CeciSeis21}, we extended the Marchioro--Pulvirenti method in two ways: First, we derived  the point vortex system from the Euler equations  for vorticity fields that admit a small \emph{unbounded}  part, in the sense that the $L^p$ norm, $p\in(2,\infty)$, of this vorticity part is sufficiently small. Second, we weakened the concentration requirement of the vortex regions. In the present study, we remove the smallness assumption on the unbounded part of the vorticity field.

Considering unbounded vorticities fields is  particularly interesting from an analytical point of view for the reason that, opposed to the situation of bounded vorticities \cite{Yudovich63}, uniqueness of solutions to the two-dimensional Euler equations is presently unknown. On the contrary, there is some evidence of the existence of more than one solution to the associated initial value problem \cite{Vishik18a, Vishik18b, BressanShen20, BressanMurray20}. Our result now shows that even in the case that the 2D Euler equations were ill-posed, in the setting under consideration the Euler equations would inherit a certain degree of uniqueness from the limiting Kirchhoff--Helmholtz point vortex system.
We will discuss this  property later in Section \ref{sec2} in more detail.

As in \cite{CeciSeis21}, we choose to work with a mild notion of vorticity concentration by considering as a measure of concentration the Wasserstein distance  $W_2$ between the vorticity field and the empirical measure associated with the  point vortices. Considering weak measures of concentration is less restrictive than supposing concentration in terms of the support if we envision chaotic or turbulent fluid motions. 
Moreover, since Wasserstein distances metrize weak convergence, see, e.g., Theorem 7.12 in \cite{Villani03}, the main result of this article  can be understood as a bound on the rate of weak convergence of Euler solutions towards the Helmholtz--Kirchhoff point vortex system.
Measures of weak concentration were used earlier in a related study of vortex filaments in three-dimensional inviscid fluids \cite{JerrardSeis17}. The Wasserstein distance that we consider here was already present in the original approach \cite{MarchioroPulvirenti83}; it was, however,   interpreted  as a second moment function (cf.~\eqref{2a} below) and played a role only on the level of intermediate results.

The method of Marchioro and Pulvirenti relies crucially  on certain symmetry properties of the two-dimensional Biot--Savart kernel and it breaks down in   situations in which  these fail to be satisfied. We are aware of two alternative approaches. One is an energy expansion method pionieered by Turkington \cite{Turkington87}, which  applies, for instance, to the axisymmetric Euler equations \cite{BenedettoCagliotiMarchioro00,ButtaMarchioro19}, which describe the evolution of vortex rings, but also to the lake equations \cite{DekeyserVanSchaftingen19}, which represent the motion of fluids whose velocity varies on scales much larger than its depth and which is smaller than the speed of gravitational waves. 
More recently, the so-called gluing method was applied successfully to the context of Euler vortex dynamics by Davila, Del Pino, Musso and Wei \cite{DavilaDelPinoMussoWei18, DavilaDelPinoMussoWei20}. The authors of these papers construct solutions in the planar two-dimensional case and for three-dimensional flows with a helical symmetry.
This approach had previously proved to be effective in other desingularization problems, see for example \cite{delPinoKowalczykWei07} and \cite{delPinoKowalzykWei11} for the Schrödinger and the Allen--Cahn equations, respectively.

Our article is organized as follows.	
In Section \ref{sec2}, we delineate the rigorous mathematical setting and present our main results. Section \ref{sec3} is devoted to the proofs.

\section{Mathematical setting}	\label{sec2}

We consider an incompressible fluid moving in the two-dimensional plane $\R^2$. The motion of the fluid can be described  by either the scalar vorticity field $\omega = \omega(t,x)\in\R$, which measures the tendency to rotate, or the vectorial velocity field $u = u(t,x)\in\R^2$. As we are interested in vortex dynamics, it is convenient to consider the two-dimensional  Euler  equations in vorticity formulation, that is,
\begin{equation}\label{e}
\partial_t \omega + u \cdot \nabla\omega = 0  \quad\text{in }(0,+\infty)\times\R^2.
\end{equation}
This transport equation states that the vorticity is advected along particle trajectories. The incompressibility of the fluid is mathematically encoded in the condition that the velocity is divergence-free,
\begin{equation}
\label{101}
\div u=0.
\end{equation}
Both vorticity and  velocity  depend on each other via the relation $\omega = \curl u = \partial_1 u_2 - \partial_2 u_1$, which can be reversed with the help of the Biot--Savart law
\begin{equation}\label{100}
u(t,x) = K*\omega(t,x) = \int_{\R^2} K(x-y)\,\omega(t,y)\,dy 
\end{equation}
(the symbol of convolution always being meant in space), and allows to recover the velocity from the vorticity. The function $K$ is the so-called Biot-Savart kernel, which  is defined as the rotated gradient of the Newtonian potential $G(z) = -\frac{1}{2\pi}\,\log|z|$, so that
\begin{equation*}
K(z) = \frac{1}{2\pi}\,\frac{z^\perp}{|z|^2}, \quad \text{where } z^\perp = (-z_2,z_1).
\end{equation*}

We equip the vorticity equation \eqref{e} with a compactly supported and possibly unbounded initial datum  $\bar{\omega} \in L^p(\R^2)$ and study the evolution of distributional solutions $\omega$ in the class $C^0((0,\infty);L^p(\R^2))$. In this paper, we have to assume that   $p>2$, which guarantees that the corresponding velocity fields, cf.~\eqref{100}, are (H\"older) continuous. Indeed, if $\omega\in L^p(\R^2)$ is compactly supported, maximal regularity estimates imply that $u$ belongs to $W^{1,p}(\R^2)$, which in turn embeds into a H\"older space. The existence of distributional solutions in our setting is known since the work of DiPerna and Majda \cite{DiPernaMajda87}.
We remark that these solutions are renormalized in the sense of DiPerna and Lions \cite{DiPernaLions89} as noticed, for instance, in \cite{LopesLopesMazzucato06}. This guarantees, in particular, the preservation in time of the Lebesgue norms. Moreover, solutions to the transport equation are Lagrangian, that is, they are advected by the Lagrangian flow associated to the fluid velocity, as in the classical setting. 

We suppose that the initial   vorticity can be decomposed into $N$ separate compactly supported \emph{components} of definite sign,
\begin{equation*}
\bar{\omega} = \sum_{i=1}^{N}\bar{\omega}_i,
\end{equation*}
where either    $\bar{\omega}_i \ge 0$ or $\bar{\omega}_i \le 0$, and
\begin{equation}\label{40}
\min_{i\neq j} \dist( \spt \bar \omega_i , \spt \bar \omega_j ) \ge \delta,
\end{equation}
for some $\delta>0$. Each individual component is now advected by the total fluid flow, and thus, its motion  is described by the linear transport equation
\begin{equation}\label{14}
\begin{cases}
\partial_t \omega_i + u \cdot \nabla\omega_i = 0, \\[.5em]
\omega_i(0,\cdot) = \bar{\omega}_i(\cdot).
\end{cases}
\end{equation}
Thanks to the Sobolev regularity of the velocity field discussed above, solutions to the linear advection equations are renormalized and unique, cf.~\cite{DiPernaLions89}, and thus, the sign of each component $\omega_i$ is preserved over time and 
the solution $\omega$ to \eqref{e} is still described by the sum
\begin{equation*}
\omega = \sum_{i=1}^{N} \omega_i.
\end{equation*}

We notice that, thanks to the incompressibility assumption \eqref{101}, the evolution equation in \eqref{14} is conservative, and thus, the total vorticity of each individual component, the so-called \emph{intensity}, is  constant in time, 
\begin{equation*}
a_i = \int_{\R^2} \bar{\omega}_i(x)\,dx = \int_{\R^2} \omega_i(t,x)\,dx.
\end{equation*}

We suppose that each vortex component is initially weakly concentrated around a single point, in the sense that there exist $N$ distinct points $\bar{Y}_i,\ldots,\bar{Y}_N \in \R^2$  such that
\begin{equation}\label{2}
W_2\left( \frac{\bar{\omega}_i}{a_i},\delta_{\bar{Y}_i} \right) \le \eps
\end{equation}
for any $i=1,\ldots,N$. Here,   $\eps>0$ is the concentration scale which we assume to be small
 and $W_2$ is the $2$-Wasserstein distance, which takes a particularly simple expression if one of the measures is atomic as it is in our setting,
\begin{equation}\label{2a}
	W_2\left( \frac{\bar{\omega}_i}{a_i},\delta_{\bar{Y}_i} \right) = \sqrt{ \frac{1}{a_i} \int_{\R^2} |x - \bar{Y}_i|^2 \,\bar{\omega}_i(x) \,dx }.
\end{equation}
We refer to Villani's monograph \cite{Villani03} for a comprehensive introduction to Wasserstein distances. At this point, we only want to make three observations.

\emph{First}, we remark that
the representation  \eqref{2a}  particularly implies that the weak concentration assumption \eqref{2} holds true whenever the vortex components are strongly concentrated in the sense that the diameter of each component is at most of the size $2\eps$. This observation was already exploited in Marchioro and Puvirenti's original work \cite{MarchioroPulvirenti83}.  In their work and in ours, the second moment function plays a central role in the analysis.

\emph{Second}, an elementary computation shows that the second moment function is minimized at the center of vorticity,
\begin{equation*}
\bar{X}_i = \frac{1}{a_i} \int_{\R^2} x \,\bar{\omega}_i(x)\,dx,
\end{equation*}
and thus, the weak concentration assumption in \eqref{2} entails that
\begin{equation}\label{3}
W_2\left( \frac{\bar{\omega}_i}{a_i},\delta_{\bar{X}_i} \right) \le W_2\left( \frac{\bar{\omega}_i}{a_i},\delta_{\bar{Y}_i} \right)\le \eps.
\end{equation}

\emph{Third}, since Wasserstein distances metrize weak convergence, see Theorem 7.12 in \cite{Villani03}, the concentration assumption \eqref{2} for the initial vorticity implies that the (rescaled) vortex component converges weakly to an atomic measure provided that $\eps\to0$. In particular, if the intensities $a_i$ are chosen independently of $\eps$, for any $p>1$, the quantities $\|\bar \omega\|_{L^p}$ have to diverge in the $\eps\to0$ limit.  In this paper, we suppose that this divergence occurs at most at an algebraic rate, 
\begin{equation}\label{5}
\|\bar{\omega}\|_{L^p} \le \Lambda \eps^{-\gamma},
\end{equation}
where $\gamma$ is some positive number and $\Lambda>0$ is a constant. 
In our earlier work \cite{CeciSeis21}, we had to restrict to exponents $\gamma \le  2(p - 2)/p$, while Caprini and Marchioro considered \eqref{5} with arbitrary $\gamma$ but $p=\infty$ \cite{CapriniMarchioro15}. Our new condition brings together the strengths  of both papers. Note that since $\omega$ and $\omega_i$ satisfy \eqref{e} and \eqref{14}, respectively, in the sense of a renormalized solution \cite{DiPernaLions89}, the scaling assumption on the intial datum \eqref{5} applies to any vortex component during the evolution,
\begin{equation}\label{104}
\|\omega_i(t)\|_{L^p} = \|\bar{\omega}_i\|_{L^p} \le \|\bar{\omega}\|_{L^p}  = \|\omega(t)\|_{L^p} .
\end{equation}

We have to make an additional assumption on the geometry of the vortex components under consideration. In order to rule out the initial vorticity to have long tentacles away from the vortex center, we suppose that there exists a radius $R$ (significantly) smaller than the separation distance $\delta$, and  large enough that the support of each initial component lies in a (open) ball of radius $R$ around its vortex center, 
\begin{equation}\label{4}
\spt\bar\omega_i \subset B_R(\bar{X}_i),
\end{equation}
for any $i\in \{1,\dots,N\}$. Such a radius can always be chosen  independently of $\eps$.

In our first main result, we show that under these concentration and scaling assumptions, the vortex component solutions   are close to the Helmholtz--Kirchhoff point-vortex system, which is given by a collection of points $Y_1,\dots, Y_N\in\R^2$, satisfying the initial value problem,
\begin{equation}\label{pvs}
\begin{cases}
\frac{d}{dt} Y_i(t) = \sum_{j\neq i} a_j K(Y_i(t)-Y_j(t)), \\[.5em]
Y_i(0) = \bar{Y}_i.
\end{cases}\quad \forall i=1,\dots,N.
\end{equation}
To be more specific, our result shows that the vortex components remain concentrated in the Wasserstein sense around the Helmholtz--Kirchhoff point vortices at the scale~$\eps$.

\begin{theorem}\label{th1}
Let $i\in\{1,\dots, N\}$ be given and suppose that  
	\[
	\eps\ll 1\quad\mbox{and}\quad R\ll \delta.
	\]
	 Then there exists a time $T>0$ independent of  $\eps $ and a positive constant $C$  such that
	\begin{equation*}
	W_2 \left( \frac{\omega_i(t)}{a_i}, \delta_{Y_i(t)} \right) \le C e^{Ct} \eps
	\end{equation*}
	for any $t\in[0,T]$. 
\end{theorem}

In the hypothesis of the theorem, we assume that the size of the vortex components is initially small compared to the distance of the components. This assumption is necessary in order to guarantee that the vortex components remain well-separated on order-one time scales. 
In particular, both geometric scales $R$ and $\delta$ are supposed to be of order one in $\eps$, which has to be sufficiently small (dependent on $\gamma$ and $p$).

For initial data that are sharply concentrated in terms of the support of the vortex components, a variation  of the analysis in \cite{CapriniMarchioro15} (along the lines of Lemma \ref{lem4} below) would also show that the vorticity remains strongly concentrated around the point vortex system. As in \cite{CapriniMarchioro15}, however, the radius of concentration would increase from $\eps$ up to $\eps^{\frac12-}$. We will not pursue this direction here but focus on a rather weak notion of vorticity concentration.


The  result from Theorem \ref{th1} translates into an estimate on the distance between the full solution $\omega $ and the singular vorticity field $\sum_i a_i \delta_{Y_i}$ if the $2$-Wasserstein distance is traded for the $1$-Wasserstein distance. The reason for choosing $W_1$ over $W_2$ is that the former can be extended to configurations that are not necessarily nonnegative. Indeed, thanks to the dual Kantorovich--Rubinstein representation 
\begin{equation*}
	W_1 (f, g) = \sup\left\{ \int_{\R^2} (f - g) \zeta \,dx : \; \|\nabla\zeta\|_{L^\infty} \le 1 \right\},
\end{equation*}
see Theorem 1.14 in \cite{Villani03}, the $1$-Wasserstein distance is well-defined if $f$ ang $g$ are integrable functions of the same global average.

 Hence, we can obtain the following corollary.

 \begin{cor}\label{cor1}
	Under the same assumptions of Theorem \ref{th1},   there exists a time $T>0$ independent of  $\eps $ and a positive constant $C$  such that
	\begin{equation*}
	W_1\left( \omega(t),\sum_i a_i \delta_{Y_i(t)} \right) \le  C e^{Ct} \eps,
	\end{equation*}
	for any $t\in[0,T]$.
\end{cor}

In particular, since Wasserstein distances metrize weak convergence in the sense of measures,   the Euler vorticity distribution $\omega(t)$ converges weakly to the point vortex measure $\sum_i a_i \delta_{Y_i(t)}$. Our result provides in addition an estimate on the convergence rate, which, up to an exponenentially in time diverging prefactor, remains of the order of $\eps$ as for the initial datum, see \eqref{2}. Furthermore, the estimate holds true for \emph{any} solution $\omega$ satisfying our initial assumptions. As we have written earlier, this is a remarkable property in light of the fact that it is currently not known whether solutions to the Euler equations with vorticity in  $L^p$  are unique, if $p<\infty$. As solutions to the point vortex system \eqref{pvs} are unique, our result shows that in any case uniqueness is in some sense recovered in the limit as $\eps \to 0$. 

We remark that the estimate in  Corollary \ref{cor1} can be interpreted as a stability estimate for the Euler equation.  Indeed, as  observed by Schochet \cite{Schochet96}, the point vortex system can be considered as a very weak solution to the Euler equation (in the sense of Delort \cite{Delort91}). Adopting this point of view, our result provides  a stability estimate between a distributional solution and an even  weaker solution, which, however, carries more  structure.
Notice that general stability estimates for the Euler equation are currently still missing. The works closest to this direction include that by Loeper \cite{Loeper06}, who also works with Wasserstein distances, see also \cite{Seis21}. In the context of linear transport equations with general Sobolev vector fields, stability estimates were obtained only recently \cite{Seis17, Seis18}.

We finally translate the estimates on the vortex components into an estimate of the vorticity centers 
\begin{equation*}
X_i(t) = \frac{1}{a_i} \int_{\R^2} x\,\omega_i(t,x)\,dx
\end{equation*}
and their velocities.

\begin{theorem}\label{th2}
	Under the same assumptions of Theorem \ref{th1},   there exists a time $T>0$ independent of  $\eps $ and a positive constant $C$  such that
	\begin{equation*}
	|X_i(t)-Y_i(t)| \le C e^{Ct} \eps, \quad \text{and} \quad \left| \frac{d}{dt}X_i(t) - \frac{d}{dt}Y_i(t) \right| \le C e^{Ct} \eps,
	\end{equation*}
	for any $t\in[0,T]$.
\end{theorem}

The result thus shows that the centers of vorticity and their velocities remain $\eps$-close to the point vortex solutions.

\section{Proofs} \label{sec3}

The strategy of our proof is heavily inspired by   Marchioro and Pulvirenti's original work \cite{MarchioroPulvirenti83}, Carprini and Marchioro's improvement \cite{CapriniMarchioro15} and builds up on our own earlier contribution   \cite{CeciSeis21}. 

We start by introducing some notation.

In the following, we write $A \lesssim B$ if an inequality  $A\le CB$ holds for some  constant $C>0$ that  is independent of $t$ and $\eps$.  We furthermore write $A \sim B$ if both $A \lesssim B$ and $B \lesssim A$ hold. This way, for instance, the vortex intensities $a_i$ will be neglected in our analysis. 
 In situations in which we choose to explicitely introduce a  constant $C$, this constant will always be independent of $t$ and $\eps$, and its value may change from line to line.

We shall also introduce  $\Omega_i(t) $ as the support of the $i$th vortex at time $t$, that is, $\Omega_i(t) = \spt \omega_i(t)$. With this notation, for any $i\not=j$, it holds that $\dist(\Omega_i(0),\Omega_j(0)) \ge \delta$, see \eqref{40}. In our main results, we claim that there exists a time $T$ that is independent of $\eps$ up to which solutions to the Euler equations are close to those of the point vortex system. Here, we make this choice more precise. In fact, we choose $T$ as the time at which the distance between the vortex components has decreased by a factor of two,
\begin{equation}\label{t1}
T= \sup\left\{ t\in(0,+\infty):\; \dist(\Omega_i(s),\Omega_j(s)) \ge \delta/2, \;\forall i\neq j, \;s\le t \right\}.
\end{equation}
Because the fluid velocity $u$   is continuous, it is clear that the vortex components cannot collide instantaneously, thus, $T>0$.  We will see later that this time is bounded from below uniformly in $\eps$, that is, $T\gtrsim1$.

Without loss of generality, we may restrict our attention to time intervals such that 
\begin{equation}\label{49}
\min_{i\neq j} |Y_i(t)-Y_j(t)| \ge \frac{\delta}{2} , 
\end{equation}	
for any $t\le T$. Indeed, if this is not the case, we select a minimal time   $\tilde T < T$ at which \eqref{49} holds with an equality. Since the point vortex system $\{Y_1,\dots,Y_N\}$ is independent of $\eps$, the time $\tilde T$ must also be independent of $\eps$, and we can work on the time interval $[0,\tilde T]$.  

We finally define by $u_i = K\ast \omega_i$ the velocity that is induced by the $i$th vortex component. For every fixed $i$, we may then decompose the total velocity  into the field induced by the $i$th vortex component and that induced by the other components. That is, we write
\begin{equation}\label{102}
	u = \sum_{j=1}^{N} K*\omega_j = u_i + F_i,
\end{equation}
where  $F_i = \sum_{j\neq i} u_j$ is the far field.

Our first step in  the proof is essentially identical to that in \cite{MarchioroPulvirenti83}. We show that until time $T$, the vorticity remains sharply concentrated in the Wasserstein sense. From this statement, our main estimates on the relation between the Euler equations and the Helmholtz--Kirchhoff system can be directly deduced. This is the content of the following Subsection \ref{SS1}. In the final Subsection \ref{SS2}, we will show that $T$ is indeed bounded from below uniformly in $\eps$, which complements the proof of our results.

\subsection{Concentration of vorticity and proofs of the main results under the assumption $T\gtrsim 1$.}\label{SS1}

The concentration estimate relies on elementary bounds on the far field.

\begin{lemma}\label{lem1}
	Let $i\in\{1,\ldots,N\}$ be given. Then for any $j\neq i$ it holds that
	\begin{equation*}
	\|u_j(t)\|_{L^{\infty}(\Omega_i(t))}\lesssim 1 \quad \mbox{and}\quad \|\grad u_j(t)\|_{L^{\infty}(\Omega_i(t))}\lesssim 1 ,
	\end{equation*}
	for any  $t\in[0,T]$.
	\end{lemma}
\begin{proof}The result is an immediate consequence of the scaling properties of the Biot--Savart kernel and the separation of the vortex components assumed up to time $T$, see \eqref{t1}. For instance, if $x\in\Omega_i(t)$ and $z\in\Omega_j(t)$,  it holds  that $|x-z| \gtrsim 1 $ and thus
	\begin{equation*}
		|u_j(t,x)| \lesssim \int \frac{1}{|x-z|}\,|\omega_j(t,z)|\,dz \lesssim 1.
	\end{equation*}
	
	Similarly, since  $K$ is Lipschitz on $B_{\delta/2}(0)^c$ with a constant of the order of $1/\delta^2 \lesssim 1$, for any $x,y\in\Omega_i(t)$, we find that
	\begin{equation*}
		|u_j(t,x)-u_j(t,y)| \le \int |K(x-z)-K(y-z)|\,|\omega_j(t,z)|\,dz \lesssim  |x-y|,
	\end{equation*}
	which is what we claimed.
\end{proof}

 Lemma \ref{lem1} establishes the Lipschitz property of the far-field, which is already sufficient to show that the vortex components remain concentrated near their vorticity centers with respect to the $2$-Wasserstein distance.

\begin{lemma}\label{lem3}
	Let $i\in\{1,\ldots,N\}$  be given. There exists a universal constant $C<\infty$   such that
	\begin{equation*}
	W_2\left(\frac{\omega_i(t)}{a_i},\,\delta_{X_i(t)}\right)\le e^{C\,t}\,\varepsilon,
	\end{equation*}
for any  $t\in[0,T]$.
\end{lemma}
This lemma is the key estimate in the Marchioro--Pulvirenti method. We recall the simple proof for the convenience of the reader. 

\begin{proof} For notational convenience, we suppose that $a_i=1$ and we set $W_i(t) = W_2(\omega_i(t),\delta_{X_i(t)})$. Using  the definition of the vortex centers $X_i$ and the splitting of the velocity field in \eqref{102}, we compute the rate of change of $W_i$,
	\begin{align*}
		\frac12\frac{d}{dt} W_i^2 &=  \int (x-X_i) \cdot u(x) \,dx 
		\\
		&=  \int (x-X_i) \cdot (F_i(x)-F_i(X_i)) \omega_i(x) \,dx +  \int (x-X_i) \cdot u_i(x)\,\omega_i(x)\,dx .
		\end{align*}
The second term on the right-hand side vanishes by the symmetry properties of the Biot--Savart kernel $K$. Using the Lipschitz bound for the far field from Lemma \ref{lem1}, we thus find the differential inequality
\[
\frac{d}{dt} W_i^2 \lesssim   W_i^2,
\]
which can be solved via a Gronwall argument, and uses the concentration of the inital datum in \eqref{3}. 
\end{proof}

We will now show that the concentration estimate implies all the estimates of our main results. We start with the proof of Theorem \ref{th2}, that is, the bound on the distance between vortex centers and point vortices and on the difference of the corresponding velocities.

\begin{proof}[Proof of Theorem \ref{th2} supposing that $T\gtrsim 1$.]
 We compute how the distance between the point vortices and the vortex centers changes under the evolutions \eqref{14} and \eqref{pvs}. We start by noticing that 
 \begin{equation}
 \label{107}
 \frac{d}{dt} X_i = \frac1{a_i}\int u\omega_i\, dx = \frac1{a_i}\sum_{j\not=i} \iint K(x-y)\omega_i(x)\omega_j(y)\, dxdy,
 \end{equation}
 thanks to the symmetry properties of the Biot--Savart kernel, and thus
	\begin{equation*}
	\begin{split}
	\left| \frac{d}{dt}|X_i - Y_i| \right| &\le \left| \frac{dX_i}{dt} - \frac{dY_i}{dt} \right| \\
	& \le \sum_{j\neq i} \left| \frac{1}{a_i} \iint K(x-y)\,\omega_i(x)\,\omega_j(y)\,dy\,dx - a_j K(Y_i - Y_j) \right| \\
	& \le \frac{1}{|a_i|} \sum_{j\neq i} \iint |K(x-y) - K(Y_i-Y_j)| \, |\omega_i(x)| \, |\omega_j(y)|\,dx\,dy, 
	\end{split}
	\end{equation*}
	as a consequence of the triangle inequality.
	Now, since both  $|x-y| \gtrsim 1$ and $|Y_i - Y_j| \gtrsim1$ in the integrand, we may use the Lipschitz property of the Biot--Savart kernel away from the origin and the triangle inequality to deduce that
	\begin{equation*}
		\begin{split}
		\left| \frac{d}{dt}|X_i - Y_i| \right| &\lesssim \int |x - Y_i| \,|\omega_i(x)|\,dx + \sum_{j\neq i} \int |y - Y_j| \,|\omega_j(y)|\,dy \\
		& \lesssim \sum_{j} \int_{\R^2} |x - X_j| \,|\omega_j(x)|\,dx + \sum_{j}|X_j - Y_j| .
		\end{split}
	\end{equation*}
We use Jensen's inequality and Lemma \ref{lem3} to bound the first term on the right-hand side. Summing over $i$, we then find the differential inequality
\[
\left| \frac{d}{dt}\sum_i |X_i - Y_i| \right|\lesssim e^{Ct}\eps + \sum_{j}|X_j - Y_j| ,
\]
and applying a standard Gronwall argument,  we deduce both statements of the theorem.	
\end{proof}

From what we have established so far, Theorem \ref{th1} is an immediate consequence.

\begin{proof}[Proof of Theorem \ref{th1} supposing that $T\gtrsim 1$.]
	The statement of Theorem \ref{th1} follows from Lemma \ref{lem3} and Theorem \ref{th2} via the triangle inequality.
\end{proof}

It remains to provide the proof of Corollary \ref{cor1}.

\begin{proof}[Proof of Corollary \ref{cor1} supposing that $T\gtrsim 1$.]

	Keeping in mind that $\omega = \sum_i \omega_i$, we have by the triangle and Jensen's inequalities that
	\begin{equation*}
		\begin{split}
		W_1 \left( \omega(t), \sum_i a_i\,\delta_{Y_i} \right) &\le \sum_i W_1 \left( \omega_i(t), a_i \delta_{Y_i(t)} \right) \\
		& \le \sum_i |a_i| W_1 \left( \frac{\omega_i(t)}{a_i}, \delta_{Y_i(t)} \right) \le \sum_i |a_i| W_2 \left( \frac{\omega_i(t)}{a_i}, \delta_{Y_i(t)} \right),
		\end{split}
	\end{equation*}
and thus, the statement  follows from Theorem \ref{th1}.
\end{proof}

\subsection{Proof of $T\gtrsim 1$.} \label{SS2}
Our goal in this subsection is to show that $T$ is bounded from below, uniformly in~$\eps$.

Before proceeding, we have to introduce some further notation. 
For any $\rho>0$ and $t\in [0,T]$, we define by
\begin{equation*}
m_i(t,\rho) = \frac{1}{a_i}\int_{B_{\rho}(X_i(t))^c} \omega_i(t,x) \,dx
\end{equation*}
the vorticity portion of the  $i$th component at time $t$ lying at least at distance $\rho$ from its center $X_i(t)$. In view of the choice of $R$ in \eqref{4}, it thus holds that $m_i(0,R)$=0.
We also remark that as a consequence of Lemma \ref{lem3}, this outer vorticity is always small in the sense that
\begin{equation}\label{108}
m_i(t,\rho) \le \frac{1}{\rho^2} W_2^2\left(\frac{\omega_i(t)}{a_i},\delta_{X_i(t)}\right) \le \frac{e^{CT}}{\rho^2}\, \eps^2.
\end{equation}
In our earlier paper \cite{CeciSeis21}, we made use of this estimate in order to conclude the uniform finiteness of $T$ in the case  $\gamma \le 2(p - 2)/p$. 

In the proof of $T\gtrsim 1$ for larger exponents $\gamma$, we distinguish two cases. 

\medskip

In the \emph{first case}, we assume that 
\begin{equation}\label{103}
m_i(t,2R) < \eps^{\alpha},\quad \mbox{for any }t\in (0,T), i\in\{1,\dots,N\},
\end{equation}
where $\alpha = p\gamma/(p-2)$. Notice that $\alpha\in (2,\infty)$ for exponents $\gamma >2(p - 2)/p$. Therefore, the condition in \eqref{103} is stronger than the one inferred above from the concentration estimate.

We show that under \eqref{103}, the velocity is bounded away from the vortex center.

\begin{lemma}\label{lem4}
Suppose that \eqref{103} holds and let $i\in\{1,\dots,N\}$ be given. Then it holds
	\begin{equation*}
		|u_i(t,x)| \lesssim 1 ,
	\end{equation*} 
for   any $t\in[0,T]$ and any $x\in \R^2$ such that  $ |x-X_i(t)|\ge 4R$.
\end{lemma}

\begin{proof}
	By dividing by the vortex intensity we may without loss of generality assume   that $\omega_i \ge 0$ and $a_i=1$ in the following. For notational convenience, we drop the time variable.
	
The velocity bound is derived from a suitable estimate of the Biot--Savart kernel. In order to separate the singular from the regular part, 	we set $L=|x-X_i|$ and decompose the velocity into two parts,
	\begin{equation*}
		|u_i(x)| \lesssim \int_{B_{L/2}(X_i)} \frac{1}{|x-y|} \,\omega_i(y)\,dy + \int_{B_{L/2}(X_i)^c} \frac{1}{|x-y|} \,\omega_i(y)\,dy = I_1 + I_2.
	\end{equation*}
	The first term is easy to bound because on the domain of integration, it holds that $|x-y|\ge L/2 \gtrsim 1$, leading to $		I_1 \lesssim  1$.
Concerning the second term, we claim that	
	\begin{equation}\label{6}
		\begin{split}
		I_2 \lesssim  \|\omega_i\|_{L^p}^{\frac{p}{2(p-1)}} m_i(t,L/2)^{\frac{p-2}{2(p-1)}},
		\end{split}
	\end{equation}
	from which the bound $I_2\lesssim 1$  follows via \eqref{5}, \eqref{104} and \eqref{103}.
	
The proof of the estimate \eqref{6} relies on a rearrangement procedure. We start by rewriting $I_2$ as
\[
I_2 = \int_{\R^2} \frac1{|x-y|}\zeta_i(y)\, dy,
\]
where $\zeta_i = \chi_{B_{\frac{L}2}(X_i)^c} \omega_i$. If $\zeta_i^*$ is the symmetric-decreasing rearrangement of $\zeta_i$ given by
\begin{equation}
\label{105}
\zeta_i^*(y) = \int_0^{\infty} \chi_{B_{\ell(s)(0)}}(y)\, ds,\quad \mbox{where }\ell(s) = \frac{|\{\zeta_i>s\}|^{\frac12}}{\pi^{\frac12}},
\end{equation}
the term $I_2$ can be bounded by
\[
I_2 \le \int_{\R^2} \frac1{|y|} \zeta_i^*(y)\, dy,
\]
see, e.g., Theorem 3.4 in \cite{LiebLoss97}. Using the explicit form of the rearrangement in \eqref{105} and Fubini's theorem, we can rewrite the latter as
\[
I_2 \le \int_0^{\infty} \int_{B_{\ell(s)}(0)} \frac1{|y|}\, dyds = 2\pi \int_0^{\infty} \ell(s)\, ds.
\]
For any $S>0$, this integral can be further estimated as
\begin{align*}
I_2  &\lesssim \int_0^S \ell(s)\, ds + \int_S^{\infty} \ell(s)\, ds \\
&\le m_i(L/2)^{\frac12} \int_0^S \frac1{s^{\frac12}}\, ds  + \|\omega_i\|_{L^p}^{\frac{p}2} \int_S^{\infty} \frac1{s^{\frac{p}2}}\, ds \\
&\lesssim S^{\frac12}  m_i(L/2)^{\frac12} + \frac1{S^{\frac{p-2}2}} \|\omega_i\|_{L^p}^{\frac{p}2}.
\end{align*}
Optimizing in $S$ yields \eqref{6}.
\end{proof}

Starting from  this bound on the velocity there  is an elementary calculation which shows that $T$ must be bounded below. 

\begin{lemma}[\cite{CeciSeis21}]\label{lem6}
Suppose that \eqref{103} holds. If $\delta\gg R$,  then $T\gtrsim 1$.
\end{lemma}

As this estimate was already proved in Lemma 11 of our previous paper, we skip its proof. 

We have completed the argument in the case where \eqref{103} holds true.

\medskip

In the \emph{second case}, we assume that 
\begin{equation}
\label{106}
m_i(t_*,2R) = \eps^{\alpha}\quad\mbox{for some }t_*\in (0,T),i\in \{1,\dots,N\},
\end{equation}
and we choose $t_*$ minimal. Notice that $t_*$ is indeed positive since $m_i(0,R)=0$ by our choice of $R$.

Instead of bounding $T$ from below, it is enough to find a bound for the smaller $t_*$. We obtain this bound from an estimate on  the outer vorticity $m_i(t_*,2R)$, which in turn relies on  an iterative procedure similar one developed by Marchioro and co-authors, see, e.g., \cite{Marchioro98_2, CapriniMarchioro15}. Differently from these works, however, we start from the hypothesis \eqref{4} that the components are initially contained in balls with a radius $R \sim 1$, while the argument in \cite{Marchioro98_2, CapriniMarchioro15} heavily exploits their assumption that the initial vorticity is strongly concentrated in balls of radius $\eps$. As already mentioned earlier, our analysis also extends to their situation and would produce strong concentration results also in the unbounded vorticity case.

Instead of estimating the outer vorticity function $m_i(t,\rho)$ directly, it is mathematically more convenient to study a smooth variant $\mu_i(t,\rho,\delta R)$, which is defined with  the help of a suitable cut-off function. More precisely, we consider a radially symmetric smooth function $\psi = \psi_{\rho,\delta R}$ satsfying
\begin{equation}\label{109}
	\psi (x)=1 \text{ if } |x|\le \rho, \quad \psi (x)=0 \text{ if } |x|>\rho+\delta R,
\end{equation} 
and
\begin{equation}\label{16}
	|\grad \psi| \lesssim \frac{1}{\delta R}, \quad |\grad^2 \psi| \lesssim \frac{1}{(\delta R)^2},
\end{equation}
and then $\mu_i$ is given by
\begin{equation}\label{17}
	\mu_i(t,\rho,\delta R)=\frac{1}{a_i}\int_{\R^2} \left(1-\psi_{\rho,\delta R}(x-X_i(t))\right)\,\omega_i(t,x)\,dx.
\end{equation}
It is readily checked that  $m_i$ and $\mu_i$ satisfy the relation
\begin{equation}\label{8}
	m_i(t,\rho+\delta R ) \le \mu_i(t,\rho,\delta R)\le m_i(t,\rho).
\end{equation}

In a first step, we derive a differential inequality for $\mu_i$. Here, we assume in addition that $T\lesssim 1$ in order to control the exponential factor in \eqref{108}. Below in the proof of Lemma \ref{lem8} we will exploit the differential inequality in order to show that also the opposite bound holds true, $T\gtrsim 1$.

\begin{lemma}\label{lem7}
Suppose that $T\lesssim 1$. Let $i=1,\ldots,N$ and  $\rho\in[R,2R]$ be given. For any $\delta R \in(0,1)$ and $t \in(0,T)$ it holds
	\begin{equation}\label{9}
		\mu_i(t,\rho,\delta R) \le \mu_i(0,\rho,\delta R) + \kappa(\eps,\delta R) \int_{0}^{t} \mu_i(s,\rho-\delta R,\delta R)\,ds,
	\end{equation}
	where
	\begin{equation*}
		\kappa(\eps,\delta R) \lesssim  \frac{\eps^2}{(\delta R)^2} + \frac{1}{\delta R} .
	\end{equation*}
\end{lemma}
\begin{proof}
	As in the previous proofs, we assume without loss of generality that $\omega_i\ge 0$ and $a_i=1$. 
	Using the fact that $\omega_i$ solves the transport equation \eqref{14}, and using the explicit formula \eqref{107} for the velocity of the centers, we compute via an integration by parts that
	\begin{equation*}
		\begin{split}
		\frac{d}{dt}\mu_i(t,\rho,\delta R) &= \frac{d}{dt}X_i \cdot \int \nabla\psi(x-X_i)\,\omega_i(x)\,dx  - \int \psi(x- X_i) \partial_t \omega_i\, dx\\
		&= -\iint K(x-y)\cdot \nabla\psi(x-X_i)\, \omega_i(y)\,\omega_i(x)\,dy\,dx \\
		&\quad + \iint \left(F_i(y)-F_i(x)\right)\cdot\nabla\psi(x-X_i)\,\omega_i(y)\,\omega_i(x)\,dy\,dx \\
		&= I_1 + I_2.
		\end{split}
	\end{equation*}
	To simplify the notation slightly, upon a shift in the spatial variables, we will assume that $X_i=0$ in what follows.
	
	We start with an estimate of $I_1$. 
	Because of the symmetry property $K(z)=-K(-z)$, using Fubini we may write
	\begin{equation*}
		I_1 = \frac12 \iint\left(\nabla\psi(y)-\nabla\psi(x)\right)\cdot K(x-y)\,\omega_i(y)\,\omega_i(x)\,dx\,dy.
	\end{equation*}
In view of the definition of the  cut-off function,  the integrand is nonzero only if  $x$ or $y$ belong to the annulus $B_{\rho+\delta R}\setminus B_{\rho}$. We split the  domain $ \R^2\times (B_{\rho+\delta R}\setminus B_{\rho}) \cup (B_{\rho+\delta R}\setminus B_{\rho})\times \R^2$ which contains the support into the following parts,
	\begin{align*}
	A = \left(B_{\rho+\delta R}\setminus B_{\rho}\right) \times B_{\frac{\rho}2},\quad
	B = B_{\frac{\rho}2} \times \left(B_{\rho+\delta R}\setminus B_{\rho}\right),\\
	C  = ((B_{\rho+\delta R}\setminus B_{\rho})\times B_{\frac{\rho}2}^c)\cup (B_{\frac{\rho}2}^c\times (B_{\rho+\delta R}\setminus B_{\rho})),
	\end{align*}
and we denote the respective contributions of $I_1$ by $I_1^A$, $I_1^B$, and $I_1^C$. 

The contributions due to $A$ and $B$ are controlled in a very similar way. For instance, 	 on the set $A$ we know that $\nabla\psi(y)=0$ and that $|x-y|\ge \rho/2$. In view of the scaling of the gradient of the cut-off function in \eqref{16}, it thus follows that
	\begin{equation*}
		\begin{split}
		|I_1^A| \lesssim \frac{1}{\rho}\, \frac{1}{\delta R} \, m_i(t,\rho) \lesssim \frac{1}{\delta R} \, m_i(t,\rho),
		\end{split}
	\end{equation*}
	by the assumption that $\rho \ge R \sim 1$.

We concentrate now on the remaining set $C$.  Here, we know that both $|x| \ge \rho/2$ and $|y| \ge \rho/2$.  For the integrand to be nonzero, we need at least one of $|x|$ or $|y|$ to be larger than $\rho$. Using the Lipschitz condition on $\nabla\psi$ guaranteed by \eqref{16}, the scaling of the Biot--Savart kernel, the bound on the outer vorticity which follows from the concentration estimate \eqref{108}, and recalling that $\rho \ge R \sim 1$, we thus  see that 
	\begin{align*}
|I_1^C| &\lesssim	\frac{1}{(\delta R)^2} \iint_C |x-y||K(x-y)| \omega_i(x)\omega_i(y)\, dxdy \\
&\lesssim \frac{1}{(\delta R)^2} m_i\left(t,\frac{\rho}{2}\right)\,m_i(t,\rho)  \\
& \lesssim \frac{\eps^2}{(\delta R)^2}\,m_i(t,\rho).
	\end{align*}
 
 By combining the previous bounds, we conclude that
 	\begin{equation*}
		|I_1| \lesssim \left( \frac{\eps^2}{(\delta R)^2} + \frac{1}{\delta R} \right) m_i(t,\rho).
	\end{equation*}

	Now we turn to the estimate of  $I_2$. Using the Lipschitz property of the far field $F_i$ on $\Omega_i$ guaranteed by Lemma \ref{lem1} and the properties of the cut-off function in \eqref{109} and    \eqref{16}, we obtain
	\[
			|I_2|\lesssim \frac{1}{\delta R} \int_{B_{\rho+\delta R} \setminus B_{\rho}} \int |x-y| \,\omega_i(y)\,dy\,\omega_i(x)\,dx.
\]
We now use 	the triangle and Jensen's inequalities and Lemma \ref{lem3} to infer
	\begin{equation*}
		\begin{split}
		|I_2|
		&\lesssim \frac{1}{\delta R} \int_{B_{\rho+\delta R} \setminus B_{\rho}} |x|\,\omega_i(x)\,dx + \frac{1}{\delta R} m_i(t,\rho)  \int|y|\,\omega_i(y)\,dy  \\
		&\lesssim \frac{\rho+\delta R}{\delta R} m_i(t,\rho) + \frac{\eps}{\delta R}  m_i(t,\rho) \\
		&\lesssim \frac{1}{\delta R} \, m_i(t,\rho).
		\end{split}
	\end{equation*}

Combining the estimates for $I_1$ and $I_2$, integrating in time and using \eqref{8} yields the thesis.
\end{proof}

The estimate in \eqref{9} is the basis  for an iteration procedure, which finally yields the desired estimate on $T$ via $t_*$.

\begin{lemma}\label{lem8}
	Suppose that \eqref{106} holds.   If $\eps\ll1$, then $T\gtrsim 1$.
\end{lemma}
\begin{proof} We argue by conradiction and suppose that for any $k,\ell\in\N$ there exists an $\eps\le 1/k$ such that $T \le 1/\ell$. In particular, it then holds that $t_*\le 1/\ell$.

Our proof is based on an iteration of estimate \eqref{9}, in which we choose $\delta R = R/M$ for some integer $M$ that we fix later. To simplify the notation in the following, we write $\mu_i(t,\rho) = \mu_i(t,\rho,R/M)$ in the following estimate.

We start by observing that  an iteration over  \eqref{9} yields
\begin{align*}
\MoveEqLeft \mu_i(t,(2-1/M)R)  \le C \left(\eps^2 M^2 +M\right) \int_0^t \mu_i(s,\left(2-2/M\right)R)\, ds\\
&\le C^{M-1} \left(\eps^2 M^2 + M\right)^{M-1} \int_0^t\int_0^{t_1}\dots\int_0^{t_{M-2}} \mu_i(t_{M-1},R)\, dt_{M-1}\dots dt_2dt_1,
\end{align*}
because the initial value terms drop out thanks to \eqref{4}. Replacing  $\mu_i$ by the outer vorticities $m_i$ with the help of \eqref{8}, and using the concentration estimate in the form of \eqref{108}, we thus obtain
\[
m_i(t,2R) \le C^{M-1} \left(\eps^2 M^2 + M\right)^{M-1}  \frac{t^{M-1}}{(M-1)!}\eps^2,
\]
upon choosing  $C$ a bit larger if necessary. For $t_*$ given by \eqref{106} and by the elementary  Stirling-type formula $n^n  < e^n n!$, the latter yields
\[
\eps^{\frac{\alpha-2}{M-1}} \lesssim \frac{\eps^2 M^2 +M}{M-1} t_*,
\]
or, equivalently,
\[
\frac{\eps^{\frac{\alpha-2}{M-1}}}{\eps^2 M+1} \lesssim t_*,
\]
for $M$ large enough. Choosing now
\[
1+ \frac{\alpha-2}{\log 2} \log\frac1{\eps}\le M \le \frac1{\eps^2}
\]
gives on the one hand that $\eps^2 M+1\le 2$,
and on the other hand it holds
$\eps^{\frac{\alpha-2}{M-1}} \ge \frac12$. Noting that such an $M$ exists provided that $\eps$ is small enough (or, equivalently, $k$ large enough, both dependent on $\gamma$ and $p$ via $\alpha$) gives that   $t_*\ge c$ for some constant $c>0$. This contradicts the upper bound $t_* \le 1/\ell$ if $\ell >1/c$. 
\end{proof}

\section*{Acknowledgments}
This work is funded by the Deutsche Forschungsgemeinschaft (DFG, German Research Foundation) under Germany's Excellence Strategy EXC 2044--390685587, Mathematics Münster: Dynamics Geometry Structure.

\bibliography{euler}
\bibliographystyle{acm}

\end{document}